\documentclass{amsart}

\usepackage[shortlabels]{enumitem}
\usepackage{amssymb,mathrsfs,dsfont}
\usepackage[latin1]{inputenc}
\usepackage[pdftex]{hyperref}
\hypersetup{colorlinks=true,pdfstartview=FitH}

\allowdisplaybreaks[4]

\newtheorem{theo}{Theorem}
\newtheorem{lemm}[theo]{Lemma}
\newtheorem{prop}[theo]{Proposition}
\newtheorem{coro}[theo]{Corollary}
\theoremstyle{definition}

\newtheorem{exam}[theo]{Example}
\theoremstyle{remark}
\newtheorem{rema}[theo]{Remark}

\newcommand{\field}[1]{\mathds{#1}}
\newcommand{\C}{\field{C}}

\newcommand{\Z}{\field{Z}}
\newcommand{\N}{\field{N}}

\newcommand{\Cir}{\field{S}^1}

\let\emptyset\varnothing

\begin{document}
\title{Dynamics of self-maps in their primal topologies}
\author{Jos\'e C. Mart\'{\i}n}
\address{Departamento de Matem\'aticas\\
Ins\-ti\-tu\-to Ve\-ne\-zo\-la\-no de Inves\-ti\-ga\-ciones Cien\-t\'{\i}\-fi\-cas\\
Apar\-ta\-do pos\-tal~20632\\
Altos de Pipe~1020-A\\
Ve\-ne\-zue\-la}
\address{Departamento de Matem\'aticas Puras y Aplicadas\\
Universidad Sim\'on Bol\'{\i}var\\
Apar\-ta\-do pos\-tal 89000\\
Caracas 1086--A\\
Ve\-ne\-zue\-la}
\email{jmartin@usb.ve}

\date{December 26, 2025}


\begin{abstract}
	We study a series of dynamical concepts for self-maps in the primal topology induced by them. Among the concepts studied are non-wandering points, limit points, recurrent points, minimal sets, transitive points and self-maps, topologically ergodic self-maps, weakly mixing self-maps, strongly mixing self-maps, Lyapunov stable self-maps, chaotic self-maps in the sense of Auslander-Yorke, chaotic self-maps in the sense of Devaney, asymptotic pairs, proximal pairs, and syndetically proximal pairs.
	
	Some results are given in the more general context of continuous self-maps in an Alexandroff topological space. We prove that a continuous self-map of an Alexandroff space is always Lyapunov stable.
\end{abstract}

\maketitle

\section{Introduction}
The purpose in this paper is to study the dynamics of a self-map in the primal topology induced by the self-map in its domain.

Let $f:X\to X$ be a self-map of the set $X$. The \emph{primal topology} induced by $f$ on $X$ is $\tau_f = \{A\subseteq X: f^{-1}(A)\subseteq A\}$. These topologies were introduced independently in \cite{sg} and \cite{echi}. In the first paper, they are called the ``functional Alexandroff topologies'', in the second one, they are called the ``primal topologies''. These two articles address some dynamic aspects of the self-map in the primal topology that it induces. Since the introduction of primal spaces, their properties have been extensively studied, see \cite{sg2,hl,lrt,lrs1,dlrt,et,lrs2,ls,mr}. However, no one has studied the dynamic aspects of the self-map in the primal topology that it induces in more depth.

A topology $\tau$ of $X$ is an \emph{Alexandroff topology} (or a \emph{principal topology}) if the intersection of every family of open sets is open. These topologies were introduced in \cite{alexandroff} with the name of ``Diskrete R\"aume''. The primal topology $\tau_f$ is an Alexandroff topology that makes $f$ a continuous map. Some of the results we obtain are given in the more general setting of continuous self-maps of Alexandroff spaces.

The Alexandroff topologies have been proven useful in many applications in computer and natural sciences, see for example \cite{khkm,kokm,wallden,sorkin,richmond}. These topologies also play a prominent role in understanding the structure of the lattice of topologies on a given set, see \cite{steiner}. There is also a characterization of the Collatz conjecture in terms of the primal topology induced by the Collatz map in the set of positive integers, see \cite{vg}.

The study of dynamical systems generated by a self-map in the primal topology induced by the self-map is a source of examples of dynamical systems in a topology that is not $T_1$ and/or uniformizable. See the following paragraph. Same, for continuous self-maps of Alexandroff spaces.

Since a $T_1$ Alexandroff space is a discrete space, in general, our topological spaces are not $T_1$. Therefore, our definitions are not necessarily the usual ones. Some of the dynamical notions we examine usually require the notion of uniformity. But the primal space induced by the self-map $f$ is uniformizable if and only if $f$ is pointwise periodic (i.e. all points are periodic by $f$, but not necessarily with a common period), see \cite[Theorem~2.10]{shhr}. Therefore, in the definition of these notions we use quasi-uniformities instead of uniformities. Every topological space admits a quasi-uniformity, see \cite{pervin} or \cite[Proposition~2.1]{fl}. For a general account about quasi-uniformities, see \cite{fl}.

In Section~\ref{sec-prel}, we give the definitions about topological dynamics that we will use in the paper. In this section, we consider continuous self-maps of topological spaces or quasi-uniform spaces. In Section~\ref{sec-alex}, we study some dynamical properties for continuous self-maps of Alexandorff spaces. We give results about quasi-periodic points, eventually periodic points, and recurrent points. We show that every continuous self-map in an Alexandroff space is Lyapunov stable under any quasi-uniformity that induces the Alexandroff topology. So there are no continuous self-maps in Alexandroff spaces that are chaotic in the sense of Auslander-Yorke. We see that the concepts of asymptotic pair, syndetically proximal pair, and proximal pair are the same when considering the finest quasi-uniformity of those that induce the Alexandroff topology.

In Section~\ref{sec-prim}, we study some dynamical properties of a self-map on the primal topology induced by this self-map. We characterize the non-wandering set. We prove that the sets of periodic points, limit points and recurrent points are equal. We prove that the minimal sets are the periodic orbits. We characterize the transitive self-maps, and topologically ergodic self-maps. We also prove that the notions of topologically ergodic, weakly mixing and strongly mixing are equivalent. Finally, we study the proximal pairs in the finest quasi-uniformity that induces the primal topology.

In this last section, we also consider an expansive self-map of the circle and compare its dynamics with the usual topology of the circle and the primal topology given by the self-map. See Example~\ref{exa-1}.

\section{Dynamical preliminaries}
\label{sec-prel}
In this section, we provide the definitions about topological dynamics that we will use in the following sections. Since our topological spaces are not $T_1$, in general, our definitions are not necessarily the usual ones. Let $X$ be a topological space, with topology $\tau$, and $f:X\to X$ a continuous self-map. Denote by $\N$, $\Z^+$, and $\Z^-$ the set of non-negative, positive, and negative integers, respectively.

Given $x\in X$, the \emph{orbit} of $x$ under $f$ is the set $\mathcal{O}(x) = \{f^k(x): k\in\N\}$. $x$ is a \emph{periodic point} of $f$ if there is $k\in\Z^+$ such that $f^k(x) = x$. In this case, we say that $\mathcal{O}(x)$ is a \emph{periodic orbit} of $f$ and say that $|\mathcal{O}(x)|$ is the \emph{period} of $x$, where $|A|$ is the cardinality of $A$. A \emph{fixed point} is a periodic point with period $1$. $x$ is an \emph{eventually periodic point} of $f$ if there is $k\in \N$ such that $f^k(x)$ is a periodic point. We define \emph{eventually fixed point} in a similar manner.

A \emph{preorbit} of $x\in X$ is a function $\hat{x}\in X^{\N}\cup \bigcup\{X^{\N_m}: m\in\Z^+\}$, such that $\hat{x}(0) = x$ and $f(\hat{x}(k)) = \hat{x}(k-1)$, for all $k\in D_{\hat{x}}\setminus\{0\}$. Where $\N_m = \{k\in\N: k\le m\}$ and $D_{\hat{x}}$ denotes the domain of $\hat{x}$. Let $\mathrm{PO}(x)$ denote the set of all preorbits of $x$. $\hat{x}\in \mathrm{PO}(x)$ is \emph{infinite} if $D_{\hat{x}} = \N$. $\hat{x}$ is \emph{complete} if it is infinite or $D_{\hat{x}} = \N_m$ and $\hat{x}(m)\notin f(X)$. In other words, $\hat{x}$ is a complete preorbit of $x$ if it is not the restriction to a strict subset of the domain of any other preorbit of $x$. It is easy to see that every preorbit of $x$ is the restriction of a complete preorbit of $x$.

Given two subsets $A$ and $B$ of $X$, we define
$$
D(A,B) = \{k\in\N: f^k(A)\cap B\neq \emptyset\}.
$$
If $A = \{x\}$ (resp. $B = \{y\}$) we can use $D(x,B)$ (resp. $D(A,y)$) instead of $D(A,B)$. If we want to highlight $f$ we put $\mathcal{O}_f(x)$ or $D_f(A,B)$ instead of $\mathcal{O}(x)$ or $D(A,B)$, respectively.

We state that $A\subseteq \field{T}$ is a \emph{syndetic subset} of $\field{T}$ if there exists $m\in\N$, such that $[k,k+m]\cap A\neq\emptyset$, for all $k\in\field{T}$, where $\field{T}$ is $\N$ or $\Z$, see \cite{gh}. $x\in X$ is an \emph{almost periodic point} of $f$ if $D(x,U)$ is a syndetic subset of $\N$, for all neighborhood $U$ of $x$.

A point $x\in X$ is almost periodic if and only if for every neighborhood $U$ of $x$ there exists $m\in \N$, such that $\mathcal{O}(x)\subseteq \bigcap\{f^{-i}(U): 0\le i\le m\}$. This affirmation is part of \cite[Lemma~4.2.1]{vries}; the proof given there applies to this general setting.

We say that $x$ is a \emph{quasi-periodic point} of $f$ if for all neighborhood $U$ of $x$ there exists $m\in\Z^+$, such that $m\N\subseteq D(x,U)$.

A point $x\in X$ is said to be a \emph{transitive point} of $f$ if for all non-void open set $U$ and $n\in\N$ there is $k\ge n$, such that $f^k(x)\in U$. In such a case, we say that $f$ is \emph{transitive}. Note that $x$ is a transitive point of $f$ if and only if $D(x,U)$ is an infinite subset in $\N$, for all non-void open set $U$.

We say that $f$ is \emph{topologically ergodic} if $D(U,V)$ is an infinite set, for all pair of non-void open sets $U$ and $V$. It is easy to see that a transitive self-map is topologically ergodic. In general, the converse is not true, see the comment that follows Theorem~\ref{the-5}.

Let $f_2:X^2\to X^2$ be the self-map defined by $f_2(x,y) = (f(x),f(y))$. We say that $f$ is \emph{weakly mixing} if $f_2$ is topologically ergodic. Since $D_{f_2}(U_1\times U_2,V_1\times V_2) = D_f(U_1,V_1)\cap D_f(U_2,V_2)$, $f$ is weakly mixing if and only if $D_f(U_1,V_1)\cap D_f(U_2,V_2)$ is an infinite set, forall quartet of non-void open sets $U_1$, $U_2$, $V_1$, and $V_2$ in $X$. We state that $f$ is \emph{strongly mixing} if $\N\setminus D(U,V)$ is finite, for all pair of non-void open sets $U$ and $V$ in $X$. The following implications follow easily from definitions
$$
\text{strongly mixing}\Rightarrow \text{weakly mixing}\Rightarrow \text{topologically ergodic.}
$$

A set $A\subseteq X$ is \emph{invariant} by $f$, or \emph{$f$-invariant}, if $f(A)\subseteq A$. A set $\Lambda\subseteq X$ is a \emph{minimal set} of $f$ if it is non-empty, closed, invariant by $f$, and it does not contain a proper subset that is non-empty, closed, and invariant by $f$.

We say that $x$ is a \emph{wandering point} of $f$ if exist a neighborhood $U$ of $x$ and $n\in\Z^+$, such that $U\cap f^k(U) = \emptyset$, for all $k\ge n$. Otherwise we say that $x$ is a \emph{non-wandering point} of $f$. We denote by $\Omega(f)$ the set of non-wandering points.

Given $x\in X$, the set
$$
\omega(x) = \bigcap\bigl\{\overline{\mathcal{O}(f^k(x))}: k\in \N\bigr\}
$$
is called the \emph{limit set} of the point $x$ under $f$. Note that $y\in\omega(x)$ if and only if there is a sequence $\{n_k\}\subseteq \N$, such that
$$
\lim_{k\to\infty} n_k = \infty\qquad\text{and}\qquad y\in \lim\{f^{n_k}(x)\}_{k\in\N},
$$
i.e. $\{f^{n_k}(x)\}_{k\in\N}$ converges to $y$.

A point $x$ is said to be a \emph{recurrent point} of $f$ if $D(x,U)$ is  an infinite set, for all neighborhood $U$ of $x$. We denote by $R(f)$ the set of recurrent points of $f$.

Let $(X,\mathcal{U})$ be a quasi-uniform space, $\tau_{\mathcal{U}}$ the topology induced by $\mathcal{U}$, see \cite[\S~1.4]{fl}, and $f:X\to X$ be a continuous self-map with respect to the topology $\tau_{\mathcal{U}}$. We say that $x\in X$ is \emph{stable} or \emph{Lyapunov stable} for $f$, or $f$ is \emph{stable} or \emph{Lyapunov stable} in $x$, if for all $U\in\mathcal{U}$ there exists $V$, neighborhood of $x$, such that $(f^k(x),f^k(y))\in U$, for all $k\in\N$ and $y\in V$. $f$ is \emph{stable} or \emph{Lyapunov stable} if it is stable at every point of $X$.

Given $U\in\mathcal{U}$ and $x\in X$, we define
$$
N(x,U) = \{y\in X: \text{$(f^k(x),f^k(y))\in U$, for all $k\in\N$}\}.
$$
Note that $f$ is Lyapunov stable in $x$ if and only if $N(x,U)$ is a neighborhood of $x$, for all $U\in\mathcal{U}$.

We state that $f$ is \emph{sensitive to initial conditions} if there exists $U\in\mathcal{U}$, such that, for all $x\in X$, $N(x,U)$ is not a neighborhood of $x$. We say that $f$ is \emph{chaotic in the sense of Auslander-Yorke}, or \emph{AY-chaotic}, whenever $f$ is transitive and sensitive to initial conditions. Note that if $f$ is Lyapunov stable at some point, then $f$ is not sensitive to initial conditions, and so it is not AY-chaotic. A self-map $f$ is said to be \emph{chaotic in the sense of Devaney}, or \emph{D-chaotic},  whenever $f$ is transitive, sensitive to initial conditions, and the periodic points are dense in $X$. By definition, D-chaotic implies AY-chaotic.

A pair of points $(x,y)\in X^2$ is said to be an \emph{asymptotic pair} of $f$ in $(X,\mathcal{U})$ if, for all $U\in \mathcal{U}$, there exists $n\in\N$, such that $(f^k(x),f^k(y))\in U$, for all $k\ge n$. Let $\mathrm{Asym}(\mathcal{U},f)$ denote the set of asymptotic pairs of $f$ in $(X,\mathcal{U})$. It is evident that $(x,y)\in \mathrm{Asym}(\mathcal{U},f)$ if and only if $\N\setminus D_{f_2}((x,y),U)$ is a finite subset of $\N$, for all $U\in\mathcal{U}$.

A pair of points $(x,y)\in X^2$ is said to be a \emph{proximal pair} of $f$ in $(X,\mathcal{U})$ if $D_{f_2}((x,y),U)$ is an infinite subset of $\N$, for all $U\in \mathcal{U}$. Let $\mathrm{Prox}(\mathcal{U},f)$ denote the set of proximal pairs of $f$ in $(X,\mathcal{U})$.

A pair of points $(x,y)$ is said to be a \emph{syndetically proximal pair} of $f$ in $(X,\mathcal{U})$ if the set $D_{f_2}((x,y),U)$ is a syndetic subset of $\N$, for all $U\in \mathcal{U}$. Let $\mathrm{SyProx}(\mathcal{U},f)$ denote the set of syndetically proximal pairs of $f$ in $(X,\mathcal{U})$. Obviously,
\begin{equation}\label{eq-1}
	\mathrm{Asym}(\mathcal{U},f)\subseteq \mathrm{SyProx}(\mathcal{U},f)\subseteq \mathrm{Prox}(\mathcal{U},f).
\end{equation}

\section{Dynamics in Alexandroff spaces}
\label{sec-alex}
In this section, we give some results on dynamical systems generated by a continuous self-map $f$ of an Alexandroff topological space $(X,\tau)$. Let $V(x)$ denote the intersection of all open sets containing $x\in X$. Then $V(x)$ is the smallest open set containing $x$, i.e., $U$ is a neighborhood of $x$ if and only if $V(x)\subseteq U$. See \cite{alexandroff} and \cite{arenas} for Alexandroff spaces.

\newpage

Let $(X,\tau)$ be an Alexandroff space and $f:X\to X$ a continuous self-map.
\begin{theo}
	Let $x\in X$. The following statements apply.
	\begin{enumerate}[label=\roman*.,ref=(\roman*)]
		\item\label{it-1} If there exists $k\in\Z^+$, such that $f^k(x)\in V(x)$, then $x$ is a quasi-periodic point of $f$.
		\item\label{it-2} If there exists $k\in\Z^+$, such that $f^k(x)\in V(x)$, and $V(x)$ is finite, then $x$ is an eventually periodic point of $f$.
		\item\label{it-3} If $x$ is a recurrent point of $f$ then it is a quasi-periodic point of $f$.
		\item\label{it-4} If $x$ is a recurrent point of $f$ and $V(x)$ is finite, then it is an eventually periodic point of $f$.
	\end{enumerate}
\end{theo}
\begin{proof}
	Let $k$ be as in the statement \ref{it-1} or \ref{it-2}. Since $f^k(x)\in V(x)$ and $f$ is continuous, it follows that $f^k(V(x))\subseteq V(f^k(x))\subseteq V(x)$. By induction we can prove that $f^{mk}(V(x))\subseteq V(x)$, for all $m\in\Z^+$. In particular, $f^{mk}(x)\in V(x)$, for all $m\in\Z^+$, then statement \ref{it-1} is true. If $V(x)$ is finite, then the orbit of $x$ by $f^k$ is finite, so $x$ is eventually periodic. This proves \ref{it-2}. \ref{it-3} and \ref{it-4} follow easily from \ref{it-1} and \ref{it-2}, respectively.
\end{proof}

Let $\mathcal{U}_\tau$ be the quasi-uniformity of $X$ given by \cite[Theorem~4.1]{arenas}. $\mathcal{U}_\tau$ has a base $\{U_\tau\}$, where
$$
U_\tau = \{(x,y)\in X^2: y\in V(x)\}.
$$
Then $\mathcal{U}_\tau = \{U\in \mathcal{P}(X^2): U_\tau\subseteq U\}$, where $\mathcal{P}(A)$ is the power set of $A$. We know that the topology induced by $\mathcal{U}_\tau$ is $\tau$, i.e. $\tau_{\mathcal{U}_\tau} = \tau$.

\begin{prop}\label{pro-1}
	The quasi-uniformity $\mathcal{U}_\tau$ is the supremum of the family of quasi-uniformities on $X$ that induces the topology $\tau$.
\end{prop}
\begin{proof}
	Let $\mathcal{U}$ be a quasi-uniformity of $X$ that induces the topology $\tau$, and $U\in\mathcal{U}$. By \cite[Proposition~1.4]{fl}, $V(x)\subseteq U(x)$, for all $x\in X$. Then $U_\tau\subseteq U$. So $\mathcal{U}_\tau$ is finer than $\mathcal{U}$, i.e. $\mathcal{U}\subseteq \mathcal{U}_\tau$.
\end{proof}

The following result says that discrete dynamical systems in Alexandroff spaces are always Lyapunov stable.
\begin{theo}\label{the-3}
	Let $f:X\to X$ be a continuous self-map of the Alexandroff space $X$, with topology $\tau$, and $\mathcal{U}$ a quasi-uniformity of $X$ that induces this topology. Then $f$ is Lyapunov stable.
\end{theo}
\begin{proof}
	Let $\mathcal{U}$ be as in the statement of the theorem. By Proposition~\ref{pro-1}, $\mathcal{U}\subseteq \mathcal{U}_\tau$. Given $U\in\mathcal{U}$ and $x\in X$ take $V(x)$ as neighborhood of $x$. Since $f$ is continuous, $f^k(y)\in f^k(V(x))\subseteq V(f^k(x))\subseteq U(f^k(x))$, for all $y\in V(x)$ and $k\in \N$. Then $f$ is Lyapunov stable.
\end{proof}

\begin{rema}
	A self-map in an Alexandroff topological space is not sensitive to initial conditions, so it is not AY-chaotic and even less so D-chaotic.
\end{rema}

\begin{theo}\label{the-2}
	Let $X$ be an Alexandroff space, with topology $\tau$, and $f:X\to X$ a continuous self-map. Then $\mathrm{Asym}(\mathcal{U}_\tau,f) = \mathrm{SyProx}(\mathcal{U}_\tau,f) = \mathrm{Prox}(\mathcal{U}_\tau,f)$.
\end{theo}
\begin{proof}
	Let $(x,y)$ be a proximal pair of $f$ in $(X,\mathcal{U}_\tau)$. Then there exists $n\in\N$ such that $(f^n(x),f^n(y))\in U_\tau$, i.e. $f^n(y)\in V(f^n(x))$. By the continuity of $f$, $f^{n+k}(y)\in f^k(V(f^n(x))\subseteq V(f^{n+k}(x))$, for all $k\in\N$. Then $\mathrm{Prox}(\mathcal{U}_\tau,f) \subseteq \mathrm{Asym}(\mathcal{U}_\tau,f)$. By \eqref{eq-1}, $\mathrm{Asym}(\mathcal{U}_\tau,f) = \mathrm{SyProx}(\mathcal{U}_\tau,f) = \mathrm{Prox}(\mathcal{U}_\tau,f)$.
\end{proof}

\section{The dynamical system generated by $f$ on $(X,\tau_f)$}
\label{sec-prim}
In this section, we will study the dynamical system generated by a self-map $f$ on the primal topological space induced by $f$.

Let $X$ be a set, $f$ a self-map of $X$, and $\tau_f$ the primal topology of $X$ induced by $f$. It is easy to prove that $\tau_f$ is an Alexandroff topology on $X$; $A\subseteq X$ is $\tau_f$-closed if and only if $f(A)\subseteq A$; in particular, $\overline{\{x\}} = \mathcal{O}(x)$, for all $x\in X$; and $f$ is $\tau_f$-continuous, see \cite{sg}. In the following, $\tau_f$ is the topology unless expressly stated.

Let $\mathcal{U}_f$ denote the quasi-uniformity $\mathcal{U}_{\tau_f}$. Since $X$ is an Alexandroff space, all the results of Section~\ref{sec-alex} apply to $f$ and $\mathcal{U}_f$. For example, $\mathcal{U}_f$ is a quasi-uniformity of $X$ that induces the topology $\tau_f$, and it is the supremum of the family of quasi-uniformities on $X$ that induces $\tau_f$. As in Section~\ref{sec-alex}, we denote by $V(x)$ the smallest open set containing $x$. By \cite{sg},
\begin{equation*}
	V(x) = \{y\in X: \text{$\exists\, k\in\N$ such that $f^k(y) = x$}\}.
\end{equation*}

The following result is an immediate consequence of the characterization of $V(x)$ given above.
\begin{lemm}\label{lem-1}
	If there exists $n\in\Z^+$ such that $f^n(x)\in V(x)$, then $x$ is a periodic point.
\end{lemm}

Our first theorem in this section is a characterization of the non-wandering set in this context.
\begin{theo}\label{the-4}
	We have that $\Omega(f) = \bigcap \{f^k(X): k\in\N\}$.
\end{theo}
\begin{proof}
	Let $x\in \bigcap \{f^k(X): k\in\N\}$. For all $k\in\N$, there exists $y_k\in X$ such that $f^k(y_k) = x$. Then $y_k\in V(x)$ for all $k$. So $V(x)\cap f^k(V(x))\neq \emptyset$, for all $k$. Then $x\in \Omega(f)$. So $\bigcap \{f^k(X): k\in\N\}\subseteq \Omega(f)$.
	
	Let $x\in \Omega(f)$. By definition, there is a sequence $\{n_k\}_{k\in\N}$ in $\N$ such that
	$$
	\lim_{k\to\infty} n_k = \infty\qquad\text{and}\qquad V(x)\cap f^{n_k}(V(x))\neq \emptyset,\ \forall\, k.
	$$
	Then for all $k\in\N$, there exists $y_k\in V(x)$ such that $f^{n_k}(y_k)\in V(x)$. So there exists $m_k\ge n_k$, such that $f^{m_k}(y_k) = x$, for all $k$. Since $X\supseteq f(X)\supseteq f^2(X)\supseteq \cdots$ and $m_k\to\infty$, as $k\to\infty$, it follows that $x\in \bigcap \{f^k(X): k\in\N\}$. Then $\Omega(f) \subseteq \bigcap \{f^k(X): k\in\N\}$.
\end{proof}

Let $\sim$ be the equivalence relation on $X$ defined by $x\sim y$ if there are $m,n\in\N$, such that $f^m(x) = f^n(y)$. In this case, we say that $x$ and $y$ are \emph{orbitally related}. By \cite[Theorem~2.1]{sg}, $X/\sim$ is the set of all connected components of $X$. The connected component of $X$ that contains $x$ is the set of points in $X$ that are orbitally related to $x$, and will be denoted by $C_x$. The following theorem contains some of the basic results given in \cite{echi}.
\begin{theo}\label{the-1}
	Let $C\in X/\sim$. The following statements apply.
	\begin{enumerate}[label=\roman*.,ref=(\roman*)]
		\item\label{it-5} $C$ is invariant by $f$.
		\item\label{it-6} $C$ cannot have two different periodic orbits.
		\item\label{it-7} If $C$ contains the periodic orbit $\mathcal{O}(x)$, then for all $y\in C$ there exists $n\in N$, such that $f^n(y) = x$. In particular, $R(f|_C) = \mathcal{O}(x)$ and $\omega(y) = \mathcal{O}(x)$, for all $y\in C$.
		\item\label{it-8} If $C$ does not contain any periodic orbits, then for all $x,y\in C$ there exists $n\in \N$, such that $f^k(y)\notin V(x)$, for all $k\ge n$. In particular, $R(f|_C) = \emptyset$ and $\omega(x) = \emptyset$, for all $x\in C$.
		\item\label{it-9} $\mathrm{R}(f) = \bigcup \{\omega(x): x\in X\}$ is the set of all periodic points of $f$.
		\item\label{it-9a} $\Lambda$ is a minimal set of $f$ if and only if it is a periodic orbit of $f$.
	\end{enumerate}
\end{theo}
\begin{proof}
	 Since $x\sim f(x)$, for all $x\in X$, it follows \ref{it-5}. Since $x\not\sim y$ if $\mathcal{O}(x)$ and $\mathcal{O}(y)$ are two distinct periodic orbits, it follows \ref{it-6}.
	 
	 Let us see \ref{it-7}. Suppose that $x\in C$ is a periodic point.  Given $y\in C$ there exist $k,m\in\N$, such that $f^k(x) = f^m(y)$. Then there exists $n\ge m$, such that $f^n(y) = x$. This proves the first affirmation of \ref{it-7}. The second one follows from the first one.
	 
	 The first statement of \ref{it-8} follows immediately from Lemma~\ref{lem-1}. The second one follows from the first one. \ref{it-9} follows from \ref{it-5}--\ref{it-8}.
	 
	 Let us see \ref{it-9a}. The sufficiency is evident. Necessity. Suppose that $\Lambda\subseteq X$ is a minimal set of $f$. Then there exists $x\in\Lambda$. Since $\overline{\{x\}} = \mathcal{O}(x)$ is a non-empty, closed, and $f$-invariant set, then $\Lambda = \mathcal{O}(x)$. Since $\mathcal{O}(f(x))\varsubsetneq \mathcal{O}(x)$ if $x$ is not a periodic point of $f$, then $\mathcal{O}(x)$ is a periodic orbit of $f$.
\end{proof}

Let us look at some examples.
\begin{exam}
	Let $X_0 = \Z\times \N$ and $f_{0,n}: X_0\to X_0$ be the self-map given by $f_{0,n}(i,j) = \bigl(i+1, [j/n]\bigr)$, where $n\in\Z^+$, $n\ge 2$ and $[x]$ denote the integer part of $x$. It is evident that $f$ is a $n$-to-one map, i.e. $|f_{0,n}^{-1}(i,j)| = n$, for all $(i,j)\in X_0$. Since $f_{0,n}^{m+i_2-i_1}(i_1,j_1) = (m+i_2,0) = f_{0,n}^m(i_2,j_2)$, if $mn > \max\{j_1,j_2\}$ and $i_2\ge i_1$, then $X_0$ is connected in the primal topology induced by $f_{0,n}$ on $X_0$.
	
	Since $f_{0,n}(X_0) = X_0$, $\Omega(f_{0,n}) = X_0$. Since $f_{0,n}$ has no periodic orbits, then
	$$
	R(f_{0,n}) = \omega(x) = \emptyset,
	$$
	for all $x\in X_0$, and $f_{0,n}$ have no minimal sets. By Theorem~\ref{the-3}, $f_{0,n}$ is Lyapunov stable, and $f$ is not AY-chaotic.
\end{exam}

\begin{exam}
	Let $X_m = \Z_m\times \N$ and $f_{m,n}: X_m\to X_m$ be the self-map given by $f_{m,n}(i,j) = \bigl((i+1) \mod m, [j/n]\bigr)$, where $m,n\in\Z^+$ and $n\ge 2$. As in the previous example, $f_{m,n}$ is an $n$-to-one map and $X_m$ is connected in the primal topology induced by $f$.
	
	Since $f_{m,n}(X_m) = X_m$ and $(0,0)$ is a periodic point of $f_{m,n}$, then
	$$
	\Omega(f_{m,n}) = X_m\quad\text{and}\quad R(f_{m,n}) = \omega(x) = \mathcal{O}(0,0) = \Z_m\times\{0\},
	$$
	for all $x\in X_m$, and $\mathcal{O}(0,0)$ is the unique minimal set of $f$. As in the previous example, $f_{m,n}$ is Lyapunov stable, and it is not AY-chaotic.
\end{exam}

The following example is the well-known expansive self-map of the circle when we consider the usual topology of the circle. With the primal topology, the circle is the disconnected union of an uncountable number of copies of the previous two examples.
\begin{exam}\label{exa-1}
	Let $\Cir = \{z\in\C: |z|=1\}$, $\tau$ be the usual topology of $\Cir$, and $f:\Cir\to \Cir$ be the self-map given by $f(z) = z^m$, for $m\in\Z^+$ and $m\ge 2$. It is well-known that, with the topology $\tau$, $\Cir$ is connected, $f$ is sensitive to initial conditions; the periodic points are dense in $\Cir$; $\omega(x) \neq \emptyset$, for all $x\in\Cir$; $f$ is transitive, i.e. there exists $x\in X$ such that $\omega(x) = \Cir$; so $f$ is D-chaotic, and all the points are non-wandering, i.e. $\Omega(f) = \Cir$, see for example \cite{kh}.
	
	With the primal topology induced by $f$, all the points are non-wandering, i.e. $\Omega(f) = X$, because $f(\Cir) = \Cir$, see Theorem~\ref{the-4}. But $\Cir$ has an uncountable number of connected components that are $f$-invariant. So, $f$ is not transitive. By Theorem~\ref{the-3}, $f$ is Lyapunov stable; so $f$ is not sensitive to initial conditions. In particular, $f$ is not AY-chaotic, and even $f$ restricted to each of the connected components of $\Cir$ is not AY-chaotic.
	
	If $x\in\Cir$ is eventually periodic, then $(C_x,f|_{C_x})$ is conjugate to $(X_k,f_{k,m})$, where $k$ is the minimal period of the only periodic orbit $\mathcal{O}(y)$ of $f$ in $C_x$. In this case, by the statement~\ref{it-7} of Theorem~\ref{the-1}, all the points in $C_x$ are eventually periodic, in fact $y\in \mathcal{O}(z)$ and $R(f|_{C_x}) = \omega(z) = \mathcal{O}(y)$, for all $z\in C_x$. If $x$ is not eventually periodic then $(C_x,f|_{C_x})$ is conjugate to $(X_0,f_{0,m})$. In this case, $R(f|_{C_x}) = \omega(z) = \emptyset$, for all $z\in C_x$.
\end{exam}

The following theorem characterizes the transitivity of $f$.
\begin{theo}
	The selft-map $f$ is transitive if and only if $X$ consists of a single periodic orbit.
\end{theo}
\begin{proof}
	The sufficiency is evident. Necessity. We know that $f$ is transitive. By \cite[Theorem~2.1]{sg}, $C$ is clopen, for all $C\in X/\sim$. In what follows, we will repeatedly use Theorem~\ref{the-1} without reference to it. $C$ is invariant, for all $C\in X/\sim$. Then $X$ is connected. If $f$ does not have a periodic orbit, then $f$ is not transitive. Then $X$ has only one periodic orbit, say $\mathcal{O}(x_0)$. Let $x_1$ be a transitive point of $f$. Then there exists $n\in \N$ such that $f^n(x_1) = x_0$. Then $f^n(\mathcal{O}(x_1)) = \mathcal{O}(x_0)$. If $y\in X\setminus \mathcal{O}(x_0)$, then $V(y)\cap \mathcal{O}(x_0) = \emptyset$. This contradicts the transitivity of $x_1$. Then $X=\mathcal{O}(x_0)$.
\end{proof}

The first statement of the following lemma is obvious. The second statement is an immediate consequence of the first one.
\begin{lemm}\label{lem-2}
	If $f(X) = X$ then every complete preorbit is an infinite preorbit and every point of $X$ has a complete preorbit. Moreover,
	$$
	V(x)\subseteq f^k(V(x)),
	$$
	for all $k\in\N$ and $x\in X$.
\end{lemm}

The following theorem characterizes the topological ergodicity of $f$.
\begin{theo}\label{the-5}
	Let $f$ be a self-map of $X\neq\emptyset$. The following statements are equivalent.
	\begin{enumerate}[label=\roman*.,ref=(\roman*)]
		\item\label{it-10} $f$ is topologically ergodic.
		\item\label{it-11} $f(X) = X$ and $V(x)\subseteq V(y)$ or $V(y)\subseteq V(x)$, for all $x,y\in X$.
		\item\label{it-12} $X = \mathcal{O}(x)\cup \hat{x}(\N)$, where $\hat{x}$ is an infinite preorbit of $x$ for some $x\in X$.
	\end{enumerate}
\end{theo}
\begin{proof}
	\ref{it-10} $\Rightarrow$ \ref{it-11}. Let us see that $f(X) = X$ by reduction to absurdity. Since $f(X)\subseteq X$, we assume that there exists $x\in X\setminus f(X)$. Then $V(x) = \{x\}$ and $f^k(x)\neq x$, for all $k\in\Z^+$. Then $D(V(x),V(x)) = \{0\}$ which contradicts that $f$ is topologically ergodic.
	
	We will see that $V(x)\subseteq V(y)$ or $V(y)\subseteq V(x)$, for all $x,y\in X$. We will prove this by reduction to absurdity. We know that in a primal topological space, $V(x)\subseteq V(y)$ or $V(y)\subseteq V(x)$ or $V(x)\cap V(y) = \emptyset$, for all $x,y\in X$, see the beginning of Section~2 of \cite{sg}. Then suppose there exist $x,y\in X$ such that $V(x)\cap V(y) = \emptyset$. Since $f^k(x)\notin V(y)$, then $f^k(V(x))\cap V(y) = \emptyset$, for all $k\in\N$, which contradicts that $f$ is topologically ergodic.
	
	\ref{it-11} $\Rightarrow$ \ref{it-10}. By Lemma~\ref{lem-2}, $V(x)\subseteq f^k(V(x))$, for all $k\in\N$ and $x\in X$. Let $U$ and $V$ be open and non-empty subsets of $X$. There exist $x\in U$ and $y\in V$. Then $V(x)\subseteq U$ and $V(y)\subseteq V$. So, $\emptyset\neq V(x)\cap V(y)\subseteq f^k(V(x))\cap V(y)\subseteq f^k(U)\cap V$, for all $k\in\N$, i.e. $D(U,V) = \N$.
	
	\ref{it-12} $\Rightarrow$ \ref{it-11} is obvious. Let us see that \ref{it-11} $\Rightarrow$ \ref{it-12}. Let $x\in X$. If $x$ is a periodic point and $X=\mathcal{O}(x)$, we are done. Otherwise, there exists $x\in X$ that is not a periodic point. So, suppose that $x$ is not a periodic point. By Lemma~\ref{lem-2}, there is an infinite preorbit of $x$, $\hat{x}$. Let us demonstrate by reduction to absurdity that $X=\mathcal{O}(x)\cup \hat{x}(\N)$. Suppose that there is $y\in X\setminus (\mathcal{O}(x)\cup \hat{x}(\N))$. Since $V(z)\subseteq V(z')$ implies that $z'\in \mathcal{O}(z)$, then $V(\hat{x}(k))\not\subseteq V(y)$, for all $k\in\N$. By hypothesis, $V(y)\subseteq V(\hat{x}(k))$, for all $k\in\N$. Let $m = \min\{k\in\N: f^k(y) = x\}$ and $n = \min\{k\in\N: f^k(y) = \hat{x}(m)\}$. Since $y\notin \mathcal{O}(x)\cup \hat{x}(\N)$, then $m,n\in\Z^+$. So, $f^n(x) = x$ which contradicts that $x$ is not a periodic point.
\end{proof}

There are three kinds of topologically ergodic self-maps with the primal topology. The first one is the self-maps that are conjugate to $f:\Z\to \Z$, given by $f(i) = i+1$. The second one is the self-maps whose domain is a periodic orbit. The third one is the self-maps that are conjugate to $f: X\to X$, where $X = \Z^-\cup\Z_m$ and $m\in\Z^+$, given by
$$
f(i) = \begin{cases} i+1, &\text{if $i\in\Z^-$;}\\ (i+1) \mod m, &\text{if $i\in\Z_m$.}\end{cases}
$$
Note that only the second kind of self-maps are transitive.

\begin{coro}
	$f$ is topologically ergodic if and only if it is weakly mixing if and only if it is strongly mixing.
\end{coro}
\begin{proof}
	By the proof of Theorem~\ref{the-5}, if $f$ is topologically ergodic then $D(U,V) = \N$ for all non-empty open sets $U$ and $V$ in $X$. So $f$ is strongly mixing and also weakly mixing.
\end{proof}

By Theorem~\ref{the-2}, the notions of proximal, asymptotic, and syndetically proximal pairs are identical in Alexandroff spaces $(X,\tau)$ when the quasi-uniformity is $\mathcal{U}_\tau$. Therefore, we will only consider the notion of proximal, for pairs of points in $(X, \mathcal{U}_f)$.

We define the following relation in $X$. $x\triangleleft y$ if there are $m,n\in\N$ such that $f^m(x) = f^n(y)$ and $m\le n$. It is evident that $\triangleleft$ is a reflexive and transitive relation. Note that, as subsets of $X^2$, $\sim\ = \triangleleft\cup \triangleleft^{-1}$.

\begin{theo}
	Let $X$ be a non-empty set, $f$ be a self-map in $X$ and $C$ a connected component of $X$. The following statements apply.
	\begin{enumerate}[label=\roman*.,ref=(\roman*)]
		\item\label{it-13} $\mathrm{Prox}(\mathcal{U}_f,f) = \triangleleft$.
		\item\label{it-14} $\triangleleft|_C =\ \sim|_C = C\times C$ if and only if there is a periodic orbit in $C$.
	\end{enumerate}
\end{theo}
\begin{proof}
	Let us prove \ref{it-13}. Suppose that $(x,y)\in\mathrm{Prox}(\mathcal{U}_f,f)$. Then there is $m\in \N$ such that $f^m(y)\in V(f^m(x))$. Therefore, there exists $n\in\N$, such that $f^{m+n}(y) = f^m(x)$. So, $x\triangleleft y$. Then $\mathrm{Prox}(\mathcal{U}_f,f) \subseteq \triangleleft$. On the other hand, if $(x,y)\in \triangleleft$ then there are $m,n\in\N$ such that $f^m(x) = f^n(y)$ and $m\le n$. Then $f_2^k(x,y)\in U_f$, for all $k\ge m$. Then $(x,y)\in\mathrm{Prox}(\mathcal{U}_f,f)$. So, $\mathrm{Prox}(\mathcal{U}_f,f) = \triangleleft$.
	
	Let us prove \ref{it-14}. By \cite[Theorem~2.1]{sg}, $\sim|_C = C\times C$. Sufficiency. We suppose that $C$ contains a periodic orbit $\mathcal{O}(z)$, with period $p$. Let $(x,y)\in C$. By Theorem~\ref{the-1}, there are $m,n\in\N$ such that $f^m(x)=z=f^n(y)=f^{n+kp}(y)$, for all $k\in\N$. Then $(x,y)\in\triangleleft|_C$. So, $C\times C\subseteq \triangleleft|_C$. By Theorem~\ref{the-1}, $\triangleleft|_C\subseteq C\times C$. Then $\triangleleft|_C = C\times C$.
	
	Necessity. We suppose that $\triangleleft|_C = C\times C$. Assume, for contradiction, that $C$ does not contain a periodic orbit. Since $X\neq \emptyset$, there exists $x\in X$. By Theorem~\ref{the-1}, $\mathcal{O}(x)\subseteq C$. By hypothesis, $f^m(x)\neq f^n(x)$ if $m\neq n$. Let $y = f(x)\in C$. If $f^m(x) = f^n(y)$ then $f^m(x) = f^{n+1}(x)$. So, $n = m-1$. Then $(x,y)\in C\times C\setminus \triangleleft|_C$. Which is a contradiction.
\end{proof}

\begin{coro}
	Let $X$ be a non-empty set and $\mathcal{U}$ a quasi-uniformity of $X$ that induces the topology $\tau_f$. Then $\triangleleft \subseteq \mathrm{Prox}(\mathcal{U}, f)$.
\end{coro}

\thanks{\emph{Acknowledgments.} I want to thank Professor Jorge Vielma for introducing me to primal spaces during his visit to the Department of Mathematics at I.V.I.C. in April 2025.

I would also like to thank Professor Víctor Sirvent for his helpful comments, which have significantly improved both the content and style of this article.}

\end{document}